\documentclass[a4paper, reqno, 11pt]{amsart}

\usepackage[english]{babel}
\usepackage{amsmath}
\usepackage{amssymb}
\usepackage{amsthm}
\usepackage{mathrsfs}
\usepackage{enumerate}
\usepackage{ifthen}
\usepackage{bbm}
\usepackage{xcolor}
\usepackage{verbatim}

\usepackage[pdftex,     
plainpages=false,   
breaklinks=true,    
colorlinks=true,
linkcolor=blue,
citecolor=blue,
pdftitle={Title},
pdfauthor={P. Bonicatto, G. Ciampa and G. Crippa}
]{hyperref}

\usepackage{graphicx}
\usepackage{pgfplots}
\usepackage{tikz} 
\usepackage{tikz-3dplot}
\usetikzlibrary{patterns, patterns.meta}
\usepackage{float}
\usepackage{mathtools}
\usepackage[font=footnotesize,labelfont=bf]{caption}
\usepackage{xfrac}
\usepackage{enumitem}
\usepackage{bm}
\usepackage{bbm}
\renewcommand{\b}{\bm b}

\newcommand{\margnote}[1]{
\ifthenelse{\boolean{shownotes}}%
{\marginpar{\raggedright\tiny\texttt{#1}}}%
{}%
}

\newcommand{\hole}[1]{
\ifthenelse{\boolean{shownotes}}%
{\begin{center} \fbox{ \rule {.25cm}{0cm}
\rule[-.1cm]{0cm}{.4cm} \parbox{.85\textwidth}{\begin{center}
\texttt{#1}\end{center}} \rule {.25cm}{0cm}}\end{center}}
{}
}
\newtheorem{thm}{Theorem}[section]

\newtheorem{proposition}[thm]{Proposition}
\newtheorem{lem}[thm]{Lemma}
\newtheorem{cor}[thm]{Corollary}
\newtheorem{rem}[thm]{Remark}

\theoremstyle{definition}
\newtheorem{defn}[thm]{Definition}
\newtheorem{definition}[thm]{Definition}

\newtheorem*{mainthm*}{Selection Theorem}
\newtheorem*{mainthmdue*}{Regularity Theorem}

 \newtheorem{remark}[thm]{Remark}

\DeclareMathOperator{\dive}{\mathrm {div}}

\newcommand{\e}{\varepsilon}		       
\newcommand{\R}{\mathbb{R}}
\newcommand{\T}{\mathbb{T}}

\newcommand{\Z}{\mathbb{Z}}

\newcommand{\de}{\mathrm{d}}

\usepackage{float}   
\usepackage{fancybox}		  
\usepackage{verbatim}

\newcommand{\schema}[1]{{\bf \sc #1}}

\numberwithin{equation}{section}

\title[Weak and parabolic solutions of advection-diffusion equations]{Weak and parabolic solutions of advection-diffusion equations with rough velocity field}
\author{Paolo Bonicatto}
\address[P.\ Bonicatto]{Mathematics Institute, University of Warwick,
	Zeeman Building, CV4 7HP Coventry, UK.}
\email{paolo.bonicatto@warwick.ac.uk}

\author{Gennaro Ciampa}
\address[G.\ Ciampa]{Dipartimento di Matematica ``Federigo Enriques", Universit\`a degli Studi di Milano, Via Cesare Saldini 50, 20133 Milano, Italy.}
\email{gennaro.ciampa@unimi.it}

\author{Gianluca Crippa}
\address[G. Crippa]{Departement Mathematik Und Informatik, Universit\"at Basel, Spiegelgasse 1, CH-4051 Basel, Switzerland.}
\email{gianluca.crippa@unibas.ch}

\begin{document}

\maketitle

\begin{abstract}
We study the Cauchy problem for the advection-diffusion equation $\partial_t u + \dive (u\b ) = \Delta u$ associated with a merely integrable divergence-free vector field $\b$ defined on the torus. We discuss existence, regularity and uniqueness results for distributional and parabolic solutions, in different regimes of integrability both for the vector field and for the initial datum. We offer an up-to-date picture of the available results scattered in the literature, and we include some original proofs. We also propose some open problems, motivated by very recent results which show ill-posedness of the equation in certain regimes of integrability via convex integration schemes.
	
			\small{
			}

\end{abstract}

\vspace{20pt}


\section{Introduction}

In this survey, we give a systematic overview of some results on the so-called \emph{advection-diffusion equation} 
\begin{equation}\label{eq:ade-intro}\tag{ADE}
	\partial_t u + \dive(u\b) = \Delta u,
\end{equation}
under general, low regularity assumptions on the (divergence-free) vector field $\b$.
This equation is one of the main building blocks in fluid-dynamics models where  $u$ is a passive scalar which is simultaneously advected (by the given velocity field $\b$) and diffused. We also remark that we consider a fixed diffusivity and we are not interested in vanishing viscosity or behaviors for small diffusivity. We refer to \cite{BCC, BN, DEIJ, LBL_CPDE, LiLuo, S21} for some recent results with degenerate viscosity coefficient.

Due to the presence of the Laplacian, \eqref{eq:ade-intro} is a second-order parabolic partial differential equation. If the vector field $\b$ is smooth, classical existence and uniqueness results are available and can be found in standard PDEs textbooks (see e.g. \cite{EV}). The problem \eqref{eq:ade-intro} has been studied also outside the smooth framework in many classical references, see e.g. \cite{EV, LA2} and the more recent \cite{LBL_book}, whose approach is intimately related to a fluid-dynamics context. We propose here a recent account of the state of the art around the well-posedness problem for \eqref{eq:ade-intro}. The main motivation behind this work lies in the several groundbreaking contributions appeared over the last few years, see e.g. \cite{MS3,MS,MS2}, which have shown \emph{ill-posedness} for \eqref{eq:ade-intro} in certain regimes by means of convex integration schemes.


\subsection*{Summary of the results and structure of the paper}

Given a vector field $\b \colon [0,T] \times \T^d \to \R^d$ on the $d$-dimensional torus $\T^d:= \R^d/\Z^d$, we study the initial value problem for the advection-diffusion equation associated with $\b$, i.e. 
\begin{equation}\label{eq:ad-intro}
	\begin{cases} 
		\partial_t u + \dive(u\b) = \Delta u \\
		u |_{t=0} = u_0, 
	\end{cases}
\end{equation}
where $u_0 \colon \T^d \to \R$ is a given initial datum.
Typically, existence results are obtained by a simple approximation argument: under global bounds on the vector field, one easily establishes energy estimates for the solutions of suitable approximate problems. Such estimates allow to apply standard weak compactness results and the linearity of the equation ensures that the weak limit is a solution to \eqref{eq:ad-intro}.
At a closer look, however, an interesting feature of \eqref{eq:ad-intro} arises: it is possible to give several, a priori different, notions of “weak” solutions and this corresponds to the fact that different a priori estimates are available for \eqref{eq:ad-intro}. This opens a wide spectrum of possibilities and taming this complicated scenario, understanding the relationships among different notions of solutions, is one of the aims of the present work.

\subsubsection*{Distributional solutions} We first deal with divergence-free vector fields $\b$, satisfying a general $L^1_t L_x^p$ integrability condition in space-time, for some $1 \le p\le \infty$. Correspondingly, we assume that the initial datum $u_0 \in L^q(\T^d)$, for some $1 \le q\le \infty$, with $\sfrac{1}{p}+\sfrac{1}{q} \le 1$. This allows to introduce distributional solutions to \eqref{eq:ad-intro}, i.e. functions $u \in L^\infty_t L^q_x$ solving the equation in the sense of distributions. Notice that a mild regularity in time of solutions is always granted for evolutionary PDEs, which allows to give a meaning to the initial condition in the Cauchy problem \eqref{eq:ad-intro}. It is then easily seen that distributional solutions always exist; yet, such a notion seems too vague and uniqueness is, in general, false.

\subsubsection*{Parabolic solutions}
The general lack of uniqueness for distributional solutions motivates the introduction of another notion of solution. Hopefully, such alternative notion will share the same existence results as the distributional ones, offering at the same time some uniqueness results. It turns out that such a notion can be used to show well posedness for fields having enough integrability. If this is the case, exploiting the divergence-free constraint one can show the basic available energy estimate for smooth solutions
\begin{equation*}
	\frac{1}{2}\int_{\T^d} |u(t,x)|^2 \,\de x + \int_0^t \int_{\T^d} |\nabla u(\tau,x)|^2 \, \de\tau\, \de x  = \frac{1}{2}\int_{\T^d} |u_0(x)|^2 \,\de x, 
\end{equation*} 
for every $t \in [0,T]$. The energy estimate entices one to look for solutions possessing $L^2$ gradient, i.e. solutions that are $H^1$ in the space variable. We therefore say that a distributional solution $u \in L^\infty_t L^q_x$ to \eqref{eq:ad-intro} is \emph{parabolic} if it holds $u \in L^2_t H^1_x$. 
Crucially, parabolic solutions carry the exact regularity needed to establish their uniqueness (under a suitable integrability assumption of the field $\b$ w.r.t. the time variable as well). This uniqueness result is proven via a well-known technique, i.e.  resorting to commutators' estimates. The $L^2_tH^1_x$ regularity of the solution allows to obtain a better control on the error one commits when considering smooth approximations of the solution. Such error (which is commonly known as \emph{commutator}) always goes to $0$ in the sense of distributions; however, in order to prove uniqueness, a better control is needed. In particular, in \cite{LBL} it is shown that the commutator for parabolic solutions converges strongly to $0$ in $L^1_{t,x}$. This is made possible by the fact that, asymptotically, the commutator is related to the quantity $\b \cdot \nabla u$ and bounds for this product can be established (for parabolic solutions $u \in L_t^2H_x^1$) if $\b \in L^2_t L^2_x$. This approach is somewhat in duality with the DiPerna-Lions' theory for the linear transport equation \cite{DPL}, where the same convergence of the commutator holds provided that $u\in L^\infty_t L^q_x$ and $\nabla \b \in L^1_t L^p_x$ satisfying $1/p+1/q\leq 1$.

\subsubsection*{A regularity result for distributional solutions}
Besides existence and uniqueness results for distributional and for parabolic solutions, a legitimate question concerns the mutual relationship between these two notions; according to our definitions, parabolic solutions cannot always be defined, but if they can, then they are always distributional. The converse implication is, in general, not true: in \cite{MS3} it is shown that there exist infinitely many distributional solutions $u\in L^\infty_t L^2_x$ to \eqref{eq:ad-intro} with a vector field $\b \in L^\infty_t L^2_x$, while the parabolic one is unique. This motivates our search for a condition that guarantees \emph{parabolic regularity} of a distributional solution. We show that, in the regime $\sfrac{1}{p} + \sfrac{1}{q} \le \sfrac 12$ (and under a $L^2$ integrability assumption of $\b$ w.r.t. time), every distributional solution is parabolic (hence, a fortiori, unique). The precise statement is the following:  
\begin{mainthmdue*}
Let $p,q \in [1, \infty)$ such that $\sfrac{1}{p}+\sfrac{1}{q}\leq \sfrac{1}{2}$. If $\b\in L_t^2L_x^p$ is a divergence-free vector field and $u\in L_t^\infty L^q_x$ is a distributional solution to \eqref{eq:ad-intro}, then $u\in L_t^2 H_x^1$.
\end{mainthmdue*} 
The proof we provide is relatively short and hinges upon a refined commutator estimate (see also Remark \ref{rem:figalli} below). We show that, in the current regime, the convergence to zero of the commutators takes place in $L^2_tH^{-1}_x$ and this is enough to obtain our regularity result (see Lemma \ref{lem:conv_comm2} for the precise commutator estimate).
We remark, en passant, that the $L^2$ integrability seems critical in our argument. Recent works have shown that, at lower integrability, a severe phenomenon of non-uniqueness may arise. In particular, using convex integration techniques, in \cite{MS3, MS, MS2} the authors constructed divergence-free vector fields $\b \in C^0_tL^p_x$, with $1 \le p < \gamma(d) < 2$, such that \eqref{eq:ad-intro} admits infinitely many solutions in the class $C^0_t H^1_x$. Here $\gamma(d) = \sfrac{2d}{d+2}$ denotes a dimensional constant, which is indeed strictly smaller than the critical exponent $2$. The situation in the intermediate regime $\gamma(d) \le p < 2$ is still open and it is the object of one question we formulate. See also \cite{BV}, where nonuniqueness of weak solutions (not necessarily in the Leray class) of the Navier-Stokes equations is shown via convex integration techniques exploiting time-intermittency, and \cite{CL21}, in which it is shown that the integrability of weak solutions plays an essential role for weak-strong uniqueness results for the Navier-Stokes equations.

Finally, we observe that also the integrability in time could play a non trivial role (in a similar spirit to e.g. \cite{CL}): it seems conceivable that non-uniqueness of parabolic solutions arises when $\b \in L^2_tL^p_x$ (instead of $\b \in C^0_tL^p_x$) for a larger class of exponents $p$.

We refer the reader to Figure \ref{fig:2d} and Figure \ref{fig:3d} for a visual summary of the results concerning advection-diffusion equations. 

\subsubsection*{A comparison with LeBris-Lions' theory of renormalized solutions}
Yet another approach to \eqref{eq:ad-intro} builds on the notion of \emph{renormalized solution}. In a nutshell, such concept allows one to define the transport term $v\b$ in a completely general framework (i.e. for any choice of exponents $p,q$) and this is achieved by prescribing that the equation in \eqref{eq:ad-intro} holds not for $u$ but for a (non-linear) function of $u$ (together with some additional assumptions on the regularity of $u$). We have opted not to pursue this direction here and we refer the reader to the monograph \cite{LBL_book} where one can find, besides the theory of bounded parabolic solutions, an extensive and comprehensive study of renormalized solutions (see in particular, \cite[Chapter 2, Remark 16]{LBL_book} for an interesting comparison between distributional and renormalized solutions).

\subsubsection*{Non-smooth diffusions} 
We conclude this introduction noticing that one could, in principle, consider also more sophisticated problems of the form \eqref{eq:ad-intro}, replacing the Laplacian in the right-hand side with other (possibly non-smooth) diffusion operators. 
A thorough study of such advection-diffusion equations with non-smooth diffusions can be found for instance in \cite{F08} (see Remark \ref{rem:figalli} below) and \cite{LBL_CPDE, LBL_book}. Some of the techniques we develop here can actually be applied also to a class of non-smooth diffusion operators and this will be the content of our forthcoming work \cite{BCC2}.
	
\begin{figure}
	\centering 
	\begin{tikzpicture}[scale=1.8]
		
		\draw[fill opacity=0.3, fill=black!20!green] (0,0) -- (3,0) -- (3,1.5) -- (0,1.5); 
		
		\draw[pattern=horizontal lines, pattern color=black] (0,0) -- (3,0) -- (0,3); 
		
		\draw[pattern=vertical lines, pattern color=blue!50] (0,0) rectangle (1.5,1.5);  
		\draw[blue!50, thick] (0,0) rectangle (1.5,1.5);

		\draw[<-,black] (.5,2.1) -- (1.4,3) node[anchor=west]{{\footnotesize $\frac{1}{p} + \frac{1}{q} \le 1$: the product $u\b$ is well defined and distributional solutions exist (Prop. \ref{prop:existence_weak_sol})}}; 
		
		\draw[<-,blue] (.75,.75) -- (2.5,2.5) node[anchor=west]{{\footnotesize $\min\{p, q\} \ge 2$: existence of parabolic solution (Prop. \ref{prop:existence_parabolic_class})}};
		
		\draw[<-,black!60!green] (2.5,.75) -- (3,2) node[anchor=west]{{\footnotesize $q \ge 2$: a-priori estimates in $L^2_t H^{1}_x$ (Rmk. \ref{rmk:ultimo})}}; 
		
		\draw[thick,black] (0,3) -- (3,0);

		\draw[thick, ->] (0,0) -- (4,0) node[anchor=north] { $\frac{1}{p}$}; 
		\draw[thick, ->] (0,0) -- (0,4) node[anchor=south ] {$\frac{1}{q}$}; 
		
		\filldraw (3,0) circle (.8pt) node[anchor=north] {\footnotesize  1}; 
		\filldraw (0,3) circle (.8pt) node[anchor=east] {\footnotesize  1};

		\filldraw (1.5,0) circle (.8pt) node[anchor=north] {\footnotesize  $\frac{1}{2}$}; 
		\filldraw (0,1.5) circle (.8pt) node[anchor=east]{\footnotesize  $\frac{1}{2}$}; 
	\end{tikzpicture}
	\caption{Visual depiction of the \emph{existence} results for distributional and parabolic solutions for vector fields $\b \in L^1_t L^p_x$ and initial datum $u_0 \in L^q$.}\label{fig:2d}
\end{figure}
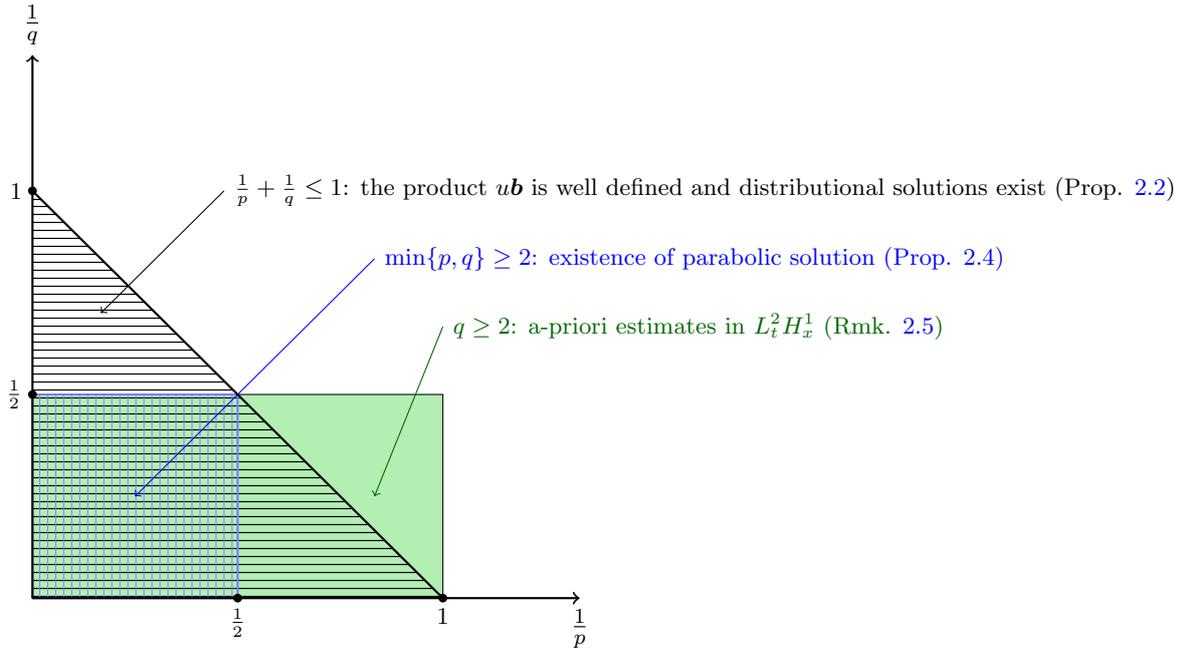 

\begin{figure}
	\centering 
	\tdplotsetmaincoords{70}{90}
	\begin{tikzpicture}[scale=4.5,tdplot_main_coords,>=latex, x={(1,-0.5,0)}]
		%
		

		\fill[fill opacity=0.3, fill=blue!99, loosely dashed]
		(0,0,.5) -- (0,.5,.5)--(.5,.5,.5)--(.5,0,.5);
		\fill[fill opacity=0.3, fill=blue!99]
		(.5,0,0) -- (.5,.5,0)--(.5,.5,.5)--(.5,0,.5);
		\fill[fill opacity=0.3, fill=blue!99]
		(0,.5,0) -- (.5,.5,0)--(.5,.5,.5)--(0,.5,.5);

		
		\fill[fill opacity=0.4, fill=red]
		(0,0,.5) -- (0,.5,0)--(.5,.5,0)--(.5,0,.5);
		\fill[fill opacity=0.4, fill=red]
		(.5,0,0) -- (.5,0,.5)--(.5,.5,0); 
		
		
		\draw[thick, blue] (0,0,.5) -- (0,.5,.5)--(.5,.5,.5)--(.5,0,.5)--cycle;
		\draw[thick, blue] (0,.5,0) -- (.5,.5,0)--(.5,.5,.5)--(0,.5,.5)--cycle;
		\draw[thick, blue] (0,0,.5) -- (0,.5,0); 
		
		\draw[thick, red] (0,0,.5) -- (0,.5,0)--(.5,.5,0)--(.5,0,.5)--cycle;
		\draw[thick, red] (.5,0,0) -- (.5,.5,0)--(.5,0,.5)--cycle;
		
		\draw[thick, black] (0,0,1) -- (0,1,0)--(1,1,0)--(1,0,1)--cycle;
		\draw[thick, black] (1,0,0) -- (1,1,0)--(1,0,1)--cycle;
		
			\draw[thick, blue] (0,0,.5) -- (.5,0,.5) -- (.5,0,0) -- (.5,.5,0) -- (0,.5,0); 
		
		\filldraw[black] (0,0,1) circle (.2pt) node[anchor=west] {\footnotesize  1}; 
		\filldraw[black] (1,0,0) circle (.2pt) node[anchor=north] {\footnotesize  1};
		\filldraw[black] (0,1,0) circle (.2pt) node[anchor=south] {\footnotesize  1};
		
		\filldraw[black] (0,0,.5) circle (.2pt) node[anchor=west] {\footnotesize $\frac{1}{2}$}; 
		\filldraw[black] (.5,0,0) circle (.2pt) node[anchor=north]  {\footnotesize $\frac{1}{2}$}; 
		\filldraw[black] (0,.5,0) circle (.2pt) node[anchor=north]  {\footnotesize $\frac{1}{2}$}; 
		
		
		\draw[<-,black] (0,.2,.7) -- (0,.3,1.1) node[anchor=south]{{\footnotesize $\frac{1}{p} +\frac{1}{q} \le 1$}};
		\draw[<-,red] (0,.15,.1) -- (0,.35,.9) node[anchor=west]{{\footnotesize $\frac{1}{p} +\frac{1}{q} \le \frac{1}{2}, \alpha\geq 2$: every distributional solution is parabolic (Thm. \ref{thm:regolarita})}};
		\draw[<-,blue] (0,.4,.2) -- (0,.5,.7) node[anchor=west]{{\footnotesize $\min\{\alpha,p,q\} \ge 2$: parabolic solutions are unique (Thm. \ref{thm:uniqueness_weak_parabolic})
		}};

		\draw[-,dashed] (0,0,0)--(1,0,0); 
		\draw[-,dashed] (0,0,0)--(0,1,0); 
		\draw[-,dashed] (0,0,0)--(0,0,1); 
		\draw[->] (0,1,0)--(0,1.5,0) node[anchor=north east]{$\frac{1}{p}$};
		\draw[->] (0,0,1)--(0,0,1.5) node[anchor=north east]{$\frac{1}{q}$}; 
		\draw[->] (1,0,0)--(1.5,0,0) node[anchor=north east]{$\frac{1}{\alpha}$}; 	
	\end{tikzpicture}
	\caption{Visual depiction of the \emph{uniqueness} and \emph{regularity} results for distributional and parabolic solutions for fields $\b \in L^\alpha_t L^p_x$ and initial datum $u_0 \in L^q$. Distributional solutions $u\in L^{\infty}_tL^q_x$ are well defined in the black wedge. In the blue cube parabolic solutions are unique and in the red wedge every distributional solution is parabolic.}\label{fig:3d}
\end{figure}
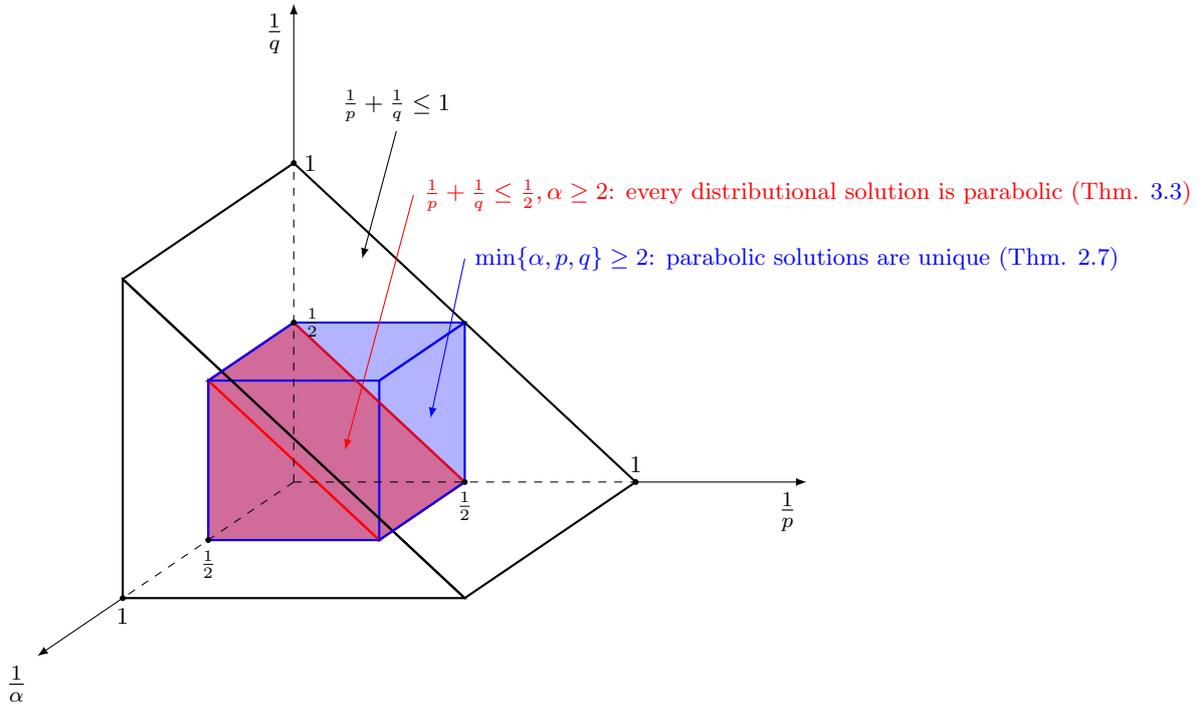

\subsection*{Notations} 
Throughout the paper, $d \ge 1$ is a fixed integer. We denote by $\T^d := \R^d/\Z^d$ the $d$-dimensional flat torus and by $\mathscr L^d$ the Lebesgue measure on it. We identify the $d$-dimensional flat torus with the cube $[0,1)^d$ and we denote with $\mathsf{d}$ the geodesic distance on $\T^d$, which is given by
$\mathsf{d}(x,y):=\min\{|x-y-k|:k\in\Z^d\,\,\mbox{such that }|k|\leq 2\}$. 
We use the letters $p,q$ to denote real numbers in $[1,+\infty]$ and $p'$ is the (H\"older) conjugate to $p$. We adopt the standard notation for Lebesgue spaces $L^p(\T^d)$ and for Sobolev spaces $W^{k,p}(\T^d)$; in particular, $H^k(\T^d) := W^{k,2}(\T^d)$. We denote by $\|\cdot\|_{L^p}$ (respectively $\|\cdot\|_{W^{k,p}}$,$\|\cdot\|_{H^k}$) the norms of the aforementioned functional spaces, omitting the domain dependence when not necessary and every definition can be adapted in a standard way to the case of spaces involving time, like e.g. $L^1([0,T];L^p(\T^d))$. 

\section{Distributional and parabolic solutions}

In this section, we are interested in the following Cauchy problem
\begin{equation}\label{eq:ad}
\begin{cases}
\partial_t u + \dive(\b u) = \Delta u &  \text{ in } (0,T) \times \T^d \\
u\vert_{t=0}=u_0 & \text{ in } \T^d
\end{cases}
\end{equation}
where the data of the problem are $T>0$, the vector field $\b$ and the initial datum $u_0$. We want first to present some different notions of solutions (distributional and parabolic) and then discuss existence, uniqueness and mutual relationship under general integrability assumptions on $\b$ and $u_0$.

\subsection{Distributional solutions}

We start by giving the following definition.
\begin{defn}[Distributional solution]\label{def:weak_sol_ad}
	Let $\b\in L^1([0,T]; L^p(\T^d))$ be a divergence-free vector field and $u_0\in L^q(\T^d)$ for $p,q$ such that $\sfrac{1}{p} +\sfrac{1}{q} \leq 1$. A function $u\in L^{\infty}([0,T];L^q(\T^d))$ is a {\em distributional solution} to \eqref{eq:ad} if for any $\varphi\in C^\infty_c([0,T)\times\T^d)$ the following equality holds: 
	$$
	\int_0^T\int_{\T^d} u(\partial_t\varphi+\b\cdot\nabla\varphi+	\Delta\varphi) \de x \de t+\int_{\T^d} u_0\varphi(0,\cdot) \de x=0.
	$$
\end{defn}
Notice that in the definition of distributional solutions the assumption that $p,q$ satisfy $\sfrac{1}{p} +\sfrac{1}{q} \leq 1$ is the minimum requirement we need in order to have $u \b\in L^1$ so that the definition makes sense. 
The proof of existence of distributional solutions is well-known and immediately follows from a classical a priori estimate.

\begin{proposition}\label{prop:existence_weak_sol}
	Let $\b\in L^1([0,T]; L^p(\T^d))$ be a divergence-free vector field and $u_0\in L^q(\T^d)$ for $p,q$ such that $\sfrac{1}{p} +\sfrac{1}{q} \leq 1$. Then there exists a distributional solution $u\in L^{\infty}([0,T];L^q(\T^d))$ to \eqref{eq:ad}.
\end{proposition}

\begin{proof}
	Let $(\rho^\delta)_\delta$ be a standard family of mollifiers and let us define $\b^\delta=\b*\rho^\delta$, $u_0^\delta=u_0*\rho^\delta$. Then, we consider the approximating problem
	\begin{equation}\label{eq:ad_be}
	\begin{cases}
	\partial_t u^\delta+\dive(\b^\delta u^\delta) =\Delta u^\delta\\
	u^\delta(0,\cdot)=u_0^\delta. 
	\end{cases}
	\end{equation}
	Being $\b^\delta$ and $u_0^\delta$ smooth, there exists a unique smooth solution $u^\delta$ to \eqref{eq:ad_be} (see \cite{EV}). It is readily checked that the sequence $u^\delta$ is equi-bounded in $L^\infty([0,T];L^q(\T^d))$. Indeed, we can multiply the equation in \eqref{eq:ad_be} by $\beta'(u^\delta)$, where $\beta \colon \R \to \R$ is a smooth, convex function: by an easy application of the chain rule and integrating in space, we get 
	\begin{equation*}
	\frac{d}{dt} \int_{\T^d} \beta(u^\delta(t, x)) \, \de x = - \int_{\T^d} \beta''(u^\delta(t, x)) |\nabla \beta (u^\delta(t, x))|^2 \, \de x \le 0.
	\end{equation*}
	In particular, fixing $t>0$ and integrating in time on $[0,T]$ we obtain 
	\begin{equation}\label{eq:estimate_beta}
	\int_{\T^d} \beta(u^\delta(t, x)) \, \de x \le \int_{\T^d} \beta(u_0^\delta(x)) \, \de x. 
	\end{equation} 
	By considering a sequence of smooth, convex functions, uniformly convergent to $\beta(s) = |s|^q$, for $1 < q < \infty$, we obtain the following uniform bounds on the $L^q$-norm of the solutions $u^\delta$:
	\begin{equation}\label{eq:diss_norm}
	\|u^\delta(t,\cdot)\|_{L^q(\T^d)}\leq\|u_0^\delta\|_{L^q(\T^d)}\leq\|u_0\|_{L^q(\T^d)}.
	\end{equation}
	For $q>1$ by standard compactness arguments, we can extract a subsequence which converges weakly-star to a function $u\in L^\infty([0,T];L^q(\T^d))$ and it is immediate to deduce that $u$ is a distributional solution of \eqref{eq:ad} because of the linearity of the equation. For $q=\infty$, the estimate \eqref{eq:diss_norm} still holds for every $\delta>0$: we send $q\to \infty$ in \eqref{eq:diss_norm} and then we can conclude as in the previous case. When $q=1$, the estimate is not sufficient to obtain weak compactness in $L^1$, as we need to show the equi-integrability of the family $(v_\e)_{\e>0}$. To do so, we argue in the following way: since $u_0^\delta \to u_0$ strongly in $L^1(\T^d)$, by De la Vall\'ee Poussin's Theorem, there exists a convex, increasing function $\Psi \colon [0,+\infty] \to [0,+\infty]$ such that $\Psi(0)=0$ and 
    \begin{equation}\label{eq:const_def}
    \lim_{s \to \infty} \frac{\Psi(s)}{s} = \infty \qquad \text{and} \qquad \sup_{\delta>0} \int_{\R^d} \Psi(|u_0^\delta(x)|) \, \de x =: C < \infty. 
    \end{equation}
    By an easy approximation argument, we can suppose $\Psi$ to be smooth and we can multiply the equation in \eqref{eq:ade} by $\Psi'(|u^\delta|)$ and we obtain
    \begin{equation*}
    \frac{d}{dt} \int_{\T^d} \Psi(|u^\delta(\tau,x)|) \, \de x  + \int_{\T^d} \Psi^{\prime\prime}(|u^\delta(\tau,x)|) |\nabla(|u^\delta|)|^2 \, \de x = 0.  
    \end{equation*}
    The convexity of $\Psi$ and an integration in time on $(0,t)$ give
    \begin{equation*}
    \int_{\T^d} \Psi(|u^\delta(t,x)|) \, \de x  \le C,
    \end{equation*}
    where $C$ is the same constant as in \eqref{eq:const_def}. Since $t$ is arbitrary, 	\begin{equation}\label{eq:bound_equi_integrability}
    \sup_{t \in (0,T)} \int_{\T^d} \Psi(|u^\delta(t,x)|) \, \de x  \le C.
    \end{equation}
    Since the constant $C$ is independent of $\delta$, we can resort again to De la Vall\'ee Poussin's Theorem and we infer that the family $(u^\delta)_{\delta>0}$ is weakly-precompact in $L^\infty([0,T]; L^1(\T^d))$, and therefore it admits a limit (up to subsequences).
\end{proof}

\subsection{Parabolic solutions}
A special sub-class of distributional solutions is given by the so-called \emph{parabolic solutions}, whose peculiar property is the Sobolev regularity in the space variable. 
As we are going to see, this notion of solution is natural for vector fields possessing enough integrability in the space variable. 

\begin{definition}
	Let $\b\in L^1([0,T]; L^2(\T^d))$ a divergence-free vector field and $u_0\in L^2(\T^d)$. A function $u\in L^{\infty}([0,T];L^2(\T^d))$ is a {\em parabolic solution} to \eqref{eq:ad} if it is a distributional solution to \eqref{eq:ad} and furthermore $u\in L^2([0,T];H^1(\T^d))$.
\end{definition}
We will sometimes refer to the space $L^2([0,T];H^1(\T^d))$ as the \emph{parabolic class}. 

\subsubsection{Existence} We now prove that, under the assumptions above, there exists at least one solution in the parabolic class:

\begin{proposition}\label{prop:existence_parabolic_class} 
	Let $\b\in L^1([0,T]; L^2(\T^d))$ be a divergence-free vector field and $u_0\in L^2(\T^d)$. Then there exists at least one parabolic solution.
\end{proposition}

\begin{proof}
	The proof follows the same idea of the one of Proposition \ref{prop:existence_weak_sol}. We consider the approximating problems \eqref{eq:ad_be} and their unique smooth solutions $u^\delta$. Choosing $\beta(s)=s^2/2$ and integrating in time on $[0,T]$, we get the following energy balance
	\begin{equation}\label{eq:energy_estimate}
	\frac{1}{2}\int_{\T^d}|u^\delta(t,x)|^2\de x+\int_0^t\int_{\T^d}|\nabla u^\delta(s,x)|^2\de x \de s=\frac{1}{2}\int_{\T^d}|u_0^\delta(x)|^2\de x. 
	\end{equation}
    The assumption $u_0\in L^2(\T^d)$ allows us to obtain a uniform estimate on the $L^2_tL^2_x$-norm of $\nabla u^\delta$: by retaining the gradient term in \eqref{eq:energy_estimate} and we use standard estimates on the convolution we obtain that
    \begin{equation}
        \|u^\delta\|_{L^\infty L^2}^2+2\|\nabla u^\delta\|_{L^2L^2}^2\leq  \|u_0\|_{L^2}^2.
    \end{equation}
	Thus, standard weak compactness arguments yield the conclusion.
\end{proof}

\begin{rem}\label{rmk:ultimo}
In the proofs of Proposition \ref{prop:existence_weak_sol} and Proposition \ref{prop:existence_parabolic_class}, we have constructed solutions as limit of solutions $(u^\delta)_{\delta}$ associated with a regularization $(\b^\delta)_{\delta}$ of the vector field and $(u_0^\delta)_{\delta}$ of the initial datum. This strategy will be used once more later in the paper and we explicitly remark here that the family $(u^\delta)_{\delta}$ satisfies two a-priori estimates:
\begin{itemize}
	\item[(E1)] $\sup_{t \in [0,T]} \| u^\delta \|_{L^q} \le  C\| u_0 \|_{L^q}$ if $u_0 \in L^q(\T^d)$; 
	\item[(E2)] $\int_0^T \| \nabla u^\delta(t,\cdot) \|_{L^2} \, \de t  \le C\| u_0 \|_{L^2}$ if $u_0 \in L^2(\T^d)$.  
\end{itemize}
These bounds follow integrating by parts and exploiting the divergence-free assumption on the vector field, in particular they are independent of the integrability of $\b$. However, we need the assumption $\b\in L^1([0,T];L^2(\T^d))$ in Proposition \ref{prop:existence_parabolic_class} to give a distributional meaning to the product $u\b$.
\end{rem}

\subsubsection{Uniqueness of solutions in the parabolic class}

The uniqueness of solutions in the parabolic class is a consequence of the following lemma, which is a straightforward modification of \cite[Lemma 5.1]{LBL}. Notice that we do not need to assume $\dive \b=0$.

\begin{lem}[Commutator estimates I]\label{lem:conv_comm}
	Consider a vector field $\b \in L^2([0,T]; L^p(\T^d))$ and a function $w \in L^\infty([0,T]; L^q(\T^d))$, where $p,q$ are positive real numbers with $\sfrac{1}{p}+\sfrac{1}{q}\leq 1$. Let $(\rho^\delta)_\delta$ be a family of smooth convolutions kernels. Define the \emph{commutator of $w$ and $\b$} as follows:
	\begin{equation}\label{eq:comm_def}
	r^\delta:=\b\cdot\nabla (w*\rho^\delta)-\left(\b\cdot\nabla w\right)*\rho^\delta.
	\end{equation}
	If $\nabla w\in L^2([0,T];L^q(\T^d))$, then $r^\delta$ converges to $0$ in $L^1([0,T] \times \T^d)$. 
\end{lem}
\begin{proof}
Observe that, for a.e. $t \in [0,T]$ and a.e. $x \in \T^d$, we can explicitly write the commutator in the following form:
	\begin{align*}
	r^\delta(t,x)& =[\b\cdot\nabla (w*\rho^\delta)](t,x)-[(\b\cdot\nabla w)*\rho^\delta](t,x)\\
	&=\b(t,x)\cdot\nabla \int_{\T^d}w(t,x-y)\rho^\delta(y)\de y-\int_{\T^d} \b(t,x-y)\cdot \nabla w(t,x-y) \rho^\delta(y)\de y \\
	&=\int_{\T^d}\rho^\delta(y) \left(\b(t,x)-\b(t,x-y)\right)\cdot\nabla w(t,x-y) \de y\\
	&=\int_{\T^d}\rho(z)\left(\b(t,x)-\b(t,x-\delta z)\right)\cdot\nabla w(t,x-\delta z) \de z.
	\end{align*}
	We thus have that
	\begin{equation*}
	\begin{split}
	\iint_{[0,T] \times \T^d}|r^\delta(t,x)|\de t \de x & =\iint_{[0,T] \times \T^d}\left| \int_{\T^d}\rho(z)\left(\b(t,x)-\b(t,x-\delta z)\right)\cdot\nabla w(t,x-\delta z) \de z\right|\de t \de x  \\
	&\leq \int_{\T^d}  \rho(z)\int_0^T \int_{\T^d}|\b(t,x)-\b(t,x-\delta z)||\nabla w(t,x-\delta z)|\de x \de t \de z.
	\end{split}
	\end{equation*}
Since $(t,x) \mapsto \b(t,x)-\b(t,x-\delta z)$ converges to $0$ in measure (for every fixed $z$), the conclusion follows by the Dominated Convergence Theorem. 
\end{proof}

Having at our disposal the previous lemma, we can now show the uniqueness of solutions in the parabolic class, arguing as in \cite{LBL}. 

\begin{thm}[Uniqueness of parabolic solutions]\label{thm:uniqueness_weak_parabolic} 
	Consider a divergence-free vector field $\b\in L^2([0,T]; L^2(\T^d))$. Then there exists at most one parabolic solution to \eqref{eq:ad}.
\end{thm}

\begin{proof} 
The uniqueness is a rather straightforward consequence of the strong convergence of commutators established in Lemma \ref{lem:conv_comm}. More precisely, since the problem is linear, it suffices to show that, if $u$ is a parabolic solution to \eqref{eq:ad} with $u_0 = 0$, then $u = 0$. Consider again a standard family of mollifiers $(\rho^\delta)_\delta$ and set $u^\delta:=u*\rho^\delta$. Then, a direct computation shows that $u^\delta$ solves the following equation
	\begin{equation}\label{eq:ade}
	\begin{cases}
	\partial_t u^\delta+\dive(\b u^\delta)=\Delta u^\delta+r^\delta\\
	u^\delta(0,\cdot)=0
	\end{cases}
	\end{equation}
where $r^\delta$ is the commutator between $u$ and $\b$, defined as in \eqref{eq:comm_def}. 
Consider now a smooth function $\beta\in C^2(\R)$, with the following properties: $\beta(s) \ge 0$, $|\beta'(s)|\le C$ for some $C>0$ and $\beta''(s) \ge 0$ for any $s\in \R$ with $\beta(s) = 0$ if, and only if, $s = 0$ (e.g. one could easily verify that the function which satisfies $\beta'(s) = \arctan(s)$ with $\beta(0)=0$ is an admissible choice).
Multiplying the equation by $\beta'(u^\delta)$ and integrating on $[0,t]\times\T^d$ we obtain
\begin{align*}
& \int_{\T^d}\beta(u^\delta) \de x+\int_0^t\int_{\T^d}\beta''(u^\delta)|\nabla u^\delta|^2\de x \de s 
= \int_0^t\int_{\T^d}\beta'(u^\delta)r^\delta\de x\de s.
\end{align*}
We now let $\delta \to 0$: using the uniform bound on $\beta'$ and Lemma \ref{lem:conv_comm} we deduce that the right-hand side converges to $0$ and thus 
\begin{equation*}
\int_{\R^d}\beta(u(t,x))\de x = - \int_0^t\int_{\T^d}\beta''(u) |\nabla u|^2\de x \de s \le 0. 
\end{equation*}
Since $t \in [0,T]$ is arbitrary, the conclusion $u\equiv 0$ easily follows.
\end{proof}

\begin{remark}Notice that in the proof above we do not really need $r^\delta \to 0$ strongly in $L^1$ as $\delta \to 0$. A weaker convergence is enough to guarantee that the integral of the product $\beta'(u^\delta)r^\delta$ converges to $0$ as $\delta \to 0$: this is precisely the idea that will be used in the next section to prove the Regularity Theorem, see in particular Lemma \ref{lem:conv_comm2}.
\end{remark}
If the vector field is less integrable than $L^2(\T^d)$, then a severe phenomenon of non-uniqueness may arise. In particular, in \cite{MS3} counterexamples are constructed via convex integration techniques: it is shown that there exist infinitely many solutions to \eqref{eq:ad} in the class $C([0,T];H^1(\T^d))$ with a divergence-free vector field $\b\in C([0,T];L^p(\T^d))$ with $1\leq p<\sfrac{2d}{d+2}$, for which it additionally holds $u\b\in L^1([0,T]\times\T^d)$. This, however, leaves open the following questions.
\begin{itemize}
\item[(Q1)] \emph{What happens in the case $\sfrac{2d}{d+2}\leq p<2$?}
\item[(Q2)] \emph{If uniqueness holds for $p$ as in (Q1), is it possible to show non-uniqueness of solutions in the larger class $L^2([0,T];H^1(\T^d))$ for vector fields which are merely $L^2$ in time (instead than continuous)?}
\item[(Q3)] \emph{For a vector field $\b\in L^r([0,T];L^2(\T^d))$ with $1\leq r<2$, are parabolic solutions unique?}
\end{itemize}
It is reasonable to imagine that a possible strategy to tackle (Q2) could be to exploit ``time-intermettency'' as in \cite{BV, CL}, which allows to increase the space integrability at the expense of the time integrability. 

\section{A regularity result} A natural question is under which conditions a distributional solution is a parabolic solution. In order to address this problem, we will need the following version of the commutator lemma which, to the best of our knowledge, is not present in the literature.

\begin{lem}[Commutator estimates II]\label{lem:conv_comm2}
	Consider a divergence-free vector field $\b \in L^2([0,T]; L^p(\T^d))$ and a function $w \in L^\infty([0,T]; L^q(\T^d))$, where $p,q$ are positive real numbers with $\sfrac{1}{p}+\sfrac{1}{q}\leq \sfrac{1}{2}$. Let $(\rho^\delta)_\delta$ be a family of smooth convoutions kernels and define $r^\delta$ as in \eqref{eq:comm_def}. Then $r^\delta$ converges to $0$ in $L^2([0,T]; H^{-1}(\T^d))$.
\end{lem}
\begin{proof}
	We write the commutator as 
	\begin{equation*}
		\begin{split}
			r^\delta & =[\b\cdot\nabla (w*\rho^\delta)]-[(\b\cdot\nabla w)*\rho^\delta] =\dive [\b (w*\rho^\delta)]-[\dive (\b w)*\rho^\delta] = \dive [\b (w*\rho^\delta) - (\b w)*\rho^\delta], 
		\end{split}
	\end{equation*}
	in the sense of distributions on $[0,T] \times \T^d$. We can thus write 
	\begin{equation*}
		r^\delta(t,x) =  \dive_x \left( \int_{\T^d} [\b(t,x) - \b(t,x-y)] w(t,x-y)\rho^\delta(y)\de y  \right),
	\end{equation*}
	and we can estimate  
	\begin{equation*}
		\begin{split}
			\|r^\delta \|_{L^2(H^{-1})}&=\sup_{\|\varphi\|_{L^2 H^1}\leq 1}\left|\iint_{[0,T] \times \T^d}r^\delta(t,x)\varphi(t,x)\de t\de x\right|\\ 
			& = \sup_{\|\varphi\|_{L^2 H^1}\leq 1} \left \vert \iint_{[0,T] \times \T^d} \left( \int_{\T^d} [\b(t,x) - \b(t,x-y)] w(t,x-y)\rho^\delta(y)\de y  \right) \nabla \varphi(t,x) \de t\de x \right \vert \\
			& \le \sup_{\|\varphi\|_{L^2 H^1}\leq 1} \int_{\T^d} \rho(z) \int_0^T \int_{\T^d}|\b(t,x)-\b(t,x-\delta z)||w(t,x-\delta z)| |\nabla\varphi(t,x)|\de x \de t \de z.  
		\end{split}
	\end{equation*}
	Notice now that, as in the proof of Lemma \ref{lem:conv_comm}, the map $(t,x) \mapsto \b(t,x)-\b(t,x-\delta z)$ converges to $0$ in measure (for every fixed $z$). H\"older inequality on the product space $[0,T] \times \T^d$ with exponents $(p,q,2)$ (in space) and $(2,\infty, 2)$ (in time) allows to apply Lebesgue Dominated Convergence Theorem and we can therefore conclude that $r^\delta \to 0$ in $L^2(H^{-1})$. 
\end{proof}

\begin{rem}[On the case $\dive \b \in L^\infty$]\label{rem:div_bounded} Lemma \ref{lem:conv_comm2} remains valid if we require only $\dive \b \in L^\infty$. Indeed, one can write the commutator $r^\delta$ in the following way: 
\begin{align*}
	r^\delta & = \dive [\b (w*\rho^\delta) - (\b w)*\rho^\delta] - (w \ast \rho^\delta) \dive \b  + (w\dive \b) \ast \rho^\delta.
\end{align*}
The first two summands converge to $0$ in $L^2H^{-1}$ precisely by the proof above. For the other two terms, one observes that $w\ast \rho^\delta \to w$ strongly in $L^\infty L^q$ (hence also in $L^\infty L^2$, since $q>2$); similarly $(w\dive \b) \ast \rho^\delta \to w \dive \b$ strongly in $L^\infty L^2$. Overall, we get that the remainder $(w\dive \b) \ast \rho^\delta-(w \ast \rho^\delta) \dive \b  \to 0$ in $L^\infty L^2$ (and thus also $L^2H^{-1}$).
\end{rem}

We can now present a regularity result which guarantees that a distributional solution in the class $L^\infty([0,T];L^q(\T^d))$ is actually in $L^2([0,T];H^1(\T^d))$ whenever $\sfrac{1}{p}+\sfrac{1}{q}\leq \sfrac{1}{2}$. 

\begin{thm}\label{thm:regolarita}
	Let $p,q \ge 1$ such that $\sfrac{1}{p}+\sfrac{1}{q}\leq \sfrac{1}{2}$. If $\b\in L^2([0,T]; L^p(\T^d))$ is a divergence-free vector field and $u\in L^\infty([0,T];L^q(\T^d))$ is a distributional solution to \eqref{eq:ad}, then $u\in L^2([0,T];H^1(\T^d))$ and satisfies 
\begin{equation}\label{eq:energy_balance}
	\frac{1}{2}\int_{\T^d}|u|^2\de x+\int_0^T\int_{\T^d}|\nabla u|^2\de x\de t = \frac{1}{2}\int_{\T^d}|u_0|^2\de x .
	\end{equation}
\end{thm}

\begin{proof} To commence, we observe that $\sfrac{1}{p}+\sfrac{1}{q}\leq \sfrac{1}{2}$ clearly implies that both $p,q \ge 2$ and, since we are on the torus, any $u\in L^\infty([0,T];L^q(\T^d))$ lies also in $L^\infty([0,T];L^2(\T^d))$. We thus need to prove $\nabla u\in L^2([0,T]; L^2(\T^d))$ and this will be achieved exhibiting an approximating sequence $(u^\delta)_{\delta}$ enjoying uniform bounds on $\nabla u^\delta$: in turn, this estimate will be obtained as a consequence of Lemma \ref{lem:conv_comm2}. 
	
Let $(\rho^\delta)_{\delta}$ be a standard family of mollifiers. As in the proof of Theorem \ref{thm:uniqueness_weak_parabolic}, the function $u^\delta:=u*\rho^\delta$ solves \eqref{eq:ade}. 
Let us now prove an estimate on the $H^1$-norm of $u^\delta$ which is independent of $\delta$: multiply the equation \eqref{eq:ade} by $u^\delta$ and integrate by parts to obtain
	\begin{equation}\label{eq:stima_energia}
	\frac{1}{2}\int_{\T^d}|u^\delta|^2\de x+\int_0^T\int_{\T^d}|\nabla u^\delta|^2\de x\de t = \frac{1}{2}\int_{\T^d}|u_0^\delta|^2\de x +\int_0^T\int_{\T^d}r^\delta u^\delta  \de x \de t.
	\end{equation}
	On the one hand, by standard properties of convolutions, we can estimate the first term in the right-hand side of \eqref{eq:stima_energia} as
	\begin{equation}\label{eq:stima_rhs1}
	\int_{\T^d}|u_0^\delta|^2\de x = \Vert u_0^\delta \Vert_{L^2}^2 \leq  \Vert u_0^\delta\Vert_{L^q}^2 \le \|u_0\|_{L^q}^2.
	\end{equation}
	On the other hand, for the second term in the right-hand side of \eqref{eq:stima_energia} we can apply Young's inequality to obtain 
	\begin{equation}\label{eq:stima_rhs2} 
	\begin{split}
	\int_0^T\int_{\T^d}r^\delta(t,x) u^\delta(t,x) \de x \de t&\leq  \|r^\delta\|_{L^2 H^{-1}} \|u^\delta\|_{L^2H^1}\\
	& \le C(T) \|r^\delta\|_{L^2 H^{-1}}^2 + \frac{1}{4(1+T)} \|u^\delta\|^2_{L^2H^1} \\ 
	& = C(T) \|r^\delta\|_{L^2 H^{-1}}^2 + \frac{1}{4(1+T)}\left( \|u^\delta\|^2_{L^2 L^2}+\|\nabla u^\delta\|^2_{L^2L^2}\right)\\
	& \le C(T) \|r^\delta\|_{L^2 H^{-1}}^2 + \frac{1}{4(1+T)}\left( T\|u^\delta\|^2_{L^\infty L^2}+\|\nabla u^\delta\|^2_{L^2L^2}\right)\\ 
	& \le C(T) \|r^\delta\|_{L^2 H^{-1}}^2 + \frac{1}{4} \left( \|u^\delta\|^2_{L^\infty L^2}+\|\nabla u^\delta\|^2_{L^2L^2}\right).
	\end{split}
	\end{equation}
	Since $r^\delta$ goes to $0$ in $L^2(H^{-1})$, the term $\|r^\delta\|_{L^2 H^{-1}}$ is equi-bounded. 
	Combining \eqref{eq:stima_rhs1}, \eqref{eq:stima_rhs2} and plugging them into \eqref{eq:stima_energia} we can conclude
	\begin{equation*}
	 \|u^\delta\|^2_{L^\infty L^2}+\|\nabla u^\delta\|^2_{L^2L^2}\leq C(T,\Vert u_0\Vert_{L^q}), 
	\end{equation*}
	for some constant $C$ which does not depend on $\delta$: this shows that the distributional solution $u$ is parabolic and thus unique thanks to Theorem \ref{thm:uniqueness_weak_parabolic}. Finally, \eqref{eq:energy_balance} immediately follows from \eqref{eq:stima_energia} sending $\delta \to 0$.
\end{proof}

\begin{rem}
    We observe that \eqref{eq:energy_balance} holds as an equality, that is, as an exact energy balance. This is a more precise information than the bound obtained in Proposition \ref{prop:existence_parabolic_class}, which is obtained via weak convergence and lower semicontinuity, and therefore holds as an inequality.
\end{rem}

Combining Theorem \ref{thm:regolarita} and Theorem \ref{thm:uniqueness_weak_parabolic}, we obtain the following corollary.
\begin{cor}
	Let $p,q \ge 1$ such that $\sfrac{1}{p}+\sfrac{1}{q}\leq \sfrac{1}{2}$. If $\b\in L^2([0,T]; L^p(\T^d))$ is a divergence-free vector field, then there exists at most one distributional solution $u\in L^\infty([0,T];L^q(\T^d))$.
\end{cor}

\subsection{Further remarks and open problems} 
We conclude this section with some observations and some further open questions. 

The assumption on the time-integrability of the vector field in Theorem \ref{thm:regolarita} suggests the following question.
\begin{itemize}
\item[(Q4)] \emph{Let $u\in L^\infty([0,T];L^q(\T^d))$ be a distributional solution associated to a divergence-free vector field $\b\in L^r([0,T];L^p(\T^d))$ with $1\leq r <2$, and assume that $\sfrac{1}{p}+\sfrac{1}{q}\leq \sfrac{1}{2}$. Is $u$ a parabolic solution?}
\end{itemize}

At this point, it is natural to wonder whether in the regime $\sfrac{1}{2} < \sfrac{1}{p}+\sfrac{1}{q}\leq 1$ there exist distributional solutions that are not parabolic and, therefore, whether uniqueness of parabolic solutions holds but uniqueness of distributional solutions does not. A partial answer to this, in dimension $d >2$, can be obtained using \cite[Theorem 1.4]{MS3}, which gives non uniqueness of distributional solutions in the regime $\sfrac{1}{p}+\sfrac{1}{q}=1$ and $p<d$ (notice that in those examples the vector field and the solution are bounded in time). A particular case of interest (somewhat reminiscent of the case of the Navier-Stokes equations in \cite{BV}) is when $p=q=2$: with such a choice, one obtains an example where there exist infinitely many distributional solutions, despite the parabolic one is unique in view of Theorem \ref{thm:uniqueness_weak_parabolic}.  

However, the convex integration schemes of \cite{MS3} are not able to cover the case $d=2$. We therefore formulate the following question. 
\begin{itemize}
\item[(Q5)]\emph{Does it exist a divergence-free vector field $\b\in L^2([0,T];L^2(\T^2))$ and a distributional solution $u\in L^\infty([0,T];L^2(\T^2))$ which is not parabolic, i.e. $u\notin L^2([0,T];H^1(\T^2))$? What if the vector field $\b\in L^2(\T^2)$ is autonomous?}
\end{itemize}
As a last point, we observe that the situation in the intermediate regime $\sfrac{1}{2} < \sfrac{1}{p}+\sfrac{1}{q}<1$ is completely open: 
\begin{itemize}
	\item[(Q6)]\emph{Let $\sfrac{1}{2} < \sfrac{1}{p}+\sfrac{1}{q}<1$. Does it exist a divergence-free vector field $\b\in L^2([0,T];L^p(\T^d))$ and a distributional solution $u\in L^\infty([0,T];L^q(\T^d))$ which is not parabolic, i.e. $u\notin L^2([0,T];H^1(\T^d))$?}
\end{itemize}

It is worth noticing that a partial answer to (Q6) is given in \cite[pag. 70]{LBL}: when $1\le q<2$ we cannot expect $\nabla u \in L^2_tL^2_x$ since this does not hold for the heat equation.



\begin{remark}\label{rem:figalli} In \cite[Theorem 4.3]{F08} it is shown that, for bounded vector fields, distributional solutions $u \in L^2_{t,x}$ are parabolic and therefore unique. This deep result is proven for advection-diffusion equations with possibly non-smooth diffusion operators and relies on a combination of parabolic PDEs tools and variational arguments. The gist of Figalli's proof of the parabolic regularity is that, for given $f \in L^2([0,T];H^{-1}(\T^d))$ and  $w_0 \in L^2(\T^d)$, there exists a unique parabolic solution to the (inhomogeneous) heat equation 
	\begin{equation*}
		\begin{cases} 
			\partial_t w - \Delta w = f\\ 
			w |_{t=0} = w_0.
		\end{cases}
	\end{equation*}
	Now given $u,\b$ as in Theorem \ref{thm:regolarita}, one can choose $f:= -\dive(\b u)$ and $w_0 := u_0$. It is then immediate to check that $w-u$ is an $L^2([0,T];L^2(\T^d))$ solution to the heat equation with zero initial data. Standard facts imply that $u=w$ a.e. and therefore parabolic regularity, uniqueness and the energy balance \eqref{eq:energy_balance} follow. Our approach, inspired by LeBris-Lions' one, is based on the new commutator estimate (Lemma \ref{lem:conv_comm2}) and gives a self-contained proof of the parabolic regularity within the theory of renormalized solutions. Interestingly, a similar $L_t^2H_x^{-1}$ commutator estimate can also be applied to non-smooth diffusions in the setting of \cite{F08} and this direction is currently under investigation in the forthcoming work \cite{BCC2}.
\end{remark}

\subsection*{Acknowledgements} P. Bonicatto has received funding from the European Research Council (ERC) under the European Union's Horizon 2020 research and innovation programme, grant agreement No 757254 (SINGULARITY). G. Ciampa is supported by the ERC Starting Grant 101039762 HamDyWWa. G. Crippa is supported by SNF Project 212573 FLUTURA – Fluids, Turbulence, Advection. Views and opinions expressed are however those of the authors only and do not necessarily reflect those of the European Union or the European Research Council. Neither the European Union nor the granting authority can be held responsible for them.

\end{document}